\documentclass[11pt, amsfonts]{amsart}

\usepackage{amsmath,amssymb,amsthm,eucal,color}
\usepackage[all]{xy}

\numberwithin{equation}{section}

\SelectTips{eu}{12}

\newtheorem{theorem}{Theorem}[section]
\newtheorem{proposition}[theorem]{Proposition}
\newtheorem{lemma}[theorem]{Lemma}
\newtheorem{corollary}[theorem]{Corollary}

\theoremstyle{definition}
\newtheorem{definition}[theorem]{Definition}
\newtheorem{example}[theorem]{Example}

\newtheorem{question}[theorem]{Question}

\theoremstyle{remark}
\newtheorem{remark}[theorem]{Remark}

\newcommand{\Q}{\mathbb{Q}}

\newcommand{\F}{\mathcal{F}}
\newcommand{\cat}{\mathrm{cat}}

\newcommand{\TC}{\mathrm{TC}}
\newcommand{\MTC}{\mathrm{MTC}}
\newcommand{\zcl}{\mathrm{zcl}}
\newcommand{\cl}{\mathrm{cup}}

\title{On the growth of topological complexity}

\author{Daisuke Kishimoto}
\address{Department of Mathematics, Kyoto University, Kyoto, 606-8502, Japan}
\email{kishi@math.kyoto-u.ac.jp}

\author{Atsushi Yamaguchi}
\address{Department of Mathematical Sciences, Osaka Prefecture University, Osaka, 599-8531, Japan}
\email{yamaguti@las.osakafu-u.ac.jp}

\subjclass[2010]{}
\keywords{topological complexity, module topological complexity, LS-category}

\begin{document}

\baselineskip.525cm

\maketitle

\begin{abstract}
  Let $\TC_r(X)$ denote the $r$-th topological complexity of a space $X$. In many cases, the generating function $\sum_{r\ge 1}\TC_{r+1}(X)x^r$ is a rational function $\frac{P(x)}{(1-x)^2}$ where $P(x)$ is a polynomial with $P(1)=\cat(X)$, that is, the asymptotic growth of $\TC_r(X)$ with respect to $r$ is $\cat(X)$. In this paper, we introduce a lower bound $\MTC_r(X)$ of $\TC_r(X)$ for a rational space $X$, and estimate the growth of $\MTC_r(X)$.
\end{abstract}


\section{Introduction}

The topological complexity $\TC(X)$ was introduced by Farber \cite{F} for the motion planning problem \cite{L,S}, which measures discontinuity of the process of robot motion planning in the configuration space $X$. The higher topological complexity $\TC_r(X)$ was introduced by Rudyak \cite{R} as a next step towards capturing the complexity of tasks that can be given to robots besides the motion problem, so that $\TC_r(X)$ measures discontinuity of the process of robot motion planning of a series of places to visit, in a specific order.

We recall the precise definition of $\TC_r(X)$. For a space $X$, let $X^r$ denote the $r$-th Cartesian product of $X$, and let $\Delta_r\colon X\to X^r$ denote the diagonal map $\Delta_r(x)=(x,x,\ldots,x)$ for $x\in X$. The $r$-th topological complexity $\TC_r(X)$ of a space $X$ is defined to be the least integer $n$ such that there is an open cover $X^r=U_0\cup U_1\cup\cdots\cup U_n$ having the property that each $U_i$ has a homotopy section $s_i\colon U_i\to X$ of $\Delta_r$, that is, $\Delta_r\circ s_i$ is homotopic to the inclusion $U_i\to X^r$. Then $\TC_2(X)$ is the topological complexity of Farber \cite{F}, and $\TC_r(X)$ for $r>2$ is the higher topological complexity of Rudyak \cite{R}. It is known that $\TC_r(X)$ is a homotopy invariant of $X$.

For a space $X$, we can define a formal power series
$$\F_X(x)=\sum_{r=1}^\infty\TC_{r+1}(X)x^r.$$
In \cite{FO}, Farber and Oprea asked the following question. Let $\cat(X)$ denote the LS-category of a space $X$.

\begin{question}
  \label{question}
  For which finite CW-complex $X$ is $\F_X(x)$ a rational function
  $$\frac{P(x)}{(1-x)^2}$$
  such that $P(x)$ is a polynomial with $P(1)=\cat(X)$?
\end{question}

As is observed in \cite{FO} (and proved in \cite{FKS}), Question \ref{question} is asking whether or not
$$\TC_{r+1}(X)=\TC_r(X)+\cat(X)$$
for all $r$ large enough. Farber, Kishimoto and Stanley \cite{FKS} proved that if $\TC_r(X)=\zcl_r(X;\Bbbk)$ and $\cat(X)=\cl(X;\Bbbk)$ for some field $\Bbbk$ and all $r$ large enough, then $\F_X(x)$ satisfies the condition in Question \ref{question}, where $\zcl_r(X;\Bbbk)$ and $\cl(X;\Bbbk)$ denote the $r$-th zero-divisors cup-length and the cup-length of $X$ over $\Bbbk$, respectively. They also showed that $\F_X(x)$ does not always satisfy the condition in Question \ref{question}.

In \cite{JMP}, Jessup, Murillo and Parent defined the module topological complexity, which is a lower bound for the topological complexity of rational spaces. We consider its higher analog, $r$-th module topological complexity $\MTC_r(X)$ for a rational space $X$. The higher topological complexity $\TC_r(X)$ is, by definition, the sectional category of the diagonal map $\Delta_r\colon X\to X^r$. In \cite{C,C1}, Carrasquel-Vela studied the sectional category of maps between rational spaces, and defined the module sectional category. We specialize the module sectional category to define $\MTC_r(X)$. We will prove:

\begin{theorem}
  \label{MTC cup}
  Let $X$ be a simply-connected rational space of finite rational type. Then for $r\ge 2$,
  $$\cl(X;\Q)\le\MTC_{r+1}(X)-\MTC_r(X)\le 2\cat(X).$$
\end{theorem}

Let $X$ be a simply-connected rational space of finite rational type. We define
$$\widehat{\F}_X(x)=\sum_{r=1}^\infty\MTC_{r+1}(X)x^r.$$
As an application of Theorem \ref{MTC cup}, one gets:

\begin{corollary}
  \label{generating MTC}
  Let $X$ be a simply-connected rational space of finite rational type.
  If $\cat(X)=\cl(X;\Q)$, then $\widehat{\F}_X(x)=\frac{P(x)}{(1-x)^2}$ where $P(x)$ is an integer polynomial with $P(1)=\cat(X)$.
\end{corollary}

If $\MTC_r(X)=\TC_r(X)$, then Theorem \ref{MTC cup} gives an estimate of the growth of $\TC_r(X)$ and Corollary \ref{generating MTC} gives a criterion for Question \ref{question}. As mentioned in \cite[Remark 5.8]{FGKV}, there is a map whose rational sectional category does not coincide with its module sectional category. However, topological complexity has a deeper connection to LS-category than general sectional category, and so in view of the celebrated result of Hess \cite{H}, one can naively expect that $\MTC_r(X)=\TC_r(X)$. Here we give a class of rational spaces, for which one has $\MTC_r(X)=\TC_r(X)$.

\begin{theorem}
  \label{odd}
  Let $X$ be a simply-connected rational space of finite rational type such that $\pi_i(X)=0$ for $i$ even. Then
  $$\MTC_r(X)=\TC_r(X)=(r-1)\cat(X).$$
\end{theorem}

There is an immediate corollary.

\begin{corollary}
  \label{odd corollary}
  Let $X$ be a simply-connected rational space of finite rational type such that $\pi_i(X)=0$ for $i$ even. Then $\F_X(x)=\frac{P(x)}{(1-x)^2}$ where $P(x)$ is an integer polynomial with $P(1)=\cat(X)$.
\end{corollary}

By applying Theorem \ref{odd}, we give a rational space $X$, to which the criterion of Farber, Kishimoto and Stanley \cite{FKS} mentioned in Section 1 does not apply, but $\F_X(x)$ satisfies the condition of Question \ref{question}.

\begin{example}
  \label{odd example}
  Let $X$ be a simply-connected rational space whose minimal model is given by
  $$\Lambda(x,y,z),\quad|x|=|y|=3,\,|z|=5,\quad dx=dy=0,\,dz=xy.$$
  Then by Corollary \ref{odd corollary}, $\F_X(x)$ satisfies the condition of Question \ref{question}. Since $\widetilde{H}^*(X;\Q)=\langle x,y,xz,yz,xyz\rangle$, one has $\widetilde{H}^*(X;\Q)^3=0$. Then it follows that $\widetilde{H}^*(X^r;\Q)^{2r+1}=0$, implying $\zcl_r(X;\Q)\le 2r$. On the other hand, the Toomer invariant of $X$ is 3, which coincides with $\cat(X)$ because $X$ is elliptic. Then by Theorem \ref{odd},
  $$\MTC_r(X)=(r-1)\cat(X)=3(r-1).$$
  In particular, $\zcl_r(X;\Q)<\MTC_r(X)$ for all $r\ge 4$.
\end{example}

\textit{Acknowledgement:} The authors would like to thank Michael Farber and John Oprea for pointing out the error in the earlier version. The first author was partly supported by JSPS KAKENHI 17K05248.

\textit{Conflict of interest:} There is not conflict of interest.


\section{Instability of the growth}

As mentioned above, Question \eqref{question} is asking whether or not $\TC_{r+1}(X)=\TC_r(X)+\cat(X)$ for all $r$ large enough. In this section, we observe why we need to consider large $r$, instead of all $r$.

Let $\Gamma$ be a finite simple graph. Let $c(\Gamma)$ be the cardinality of the maximal clique of $\Gamma$, and let $z_r(\Gamma)$ be the maximum of $\sum_{i=1}^r|C_i|$ where $C_1,\ldots,C_r$ are cliques of $\Gamma$ such that $\bigcap_{i=1}^rC_i=\emptyset$.

\begin{example}
  \label{graph}
  Let $\Gamma$ be the following graph.
  
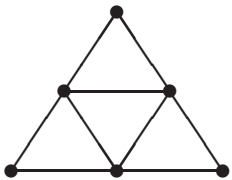
\begin{center}
    \begin{picture}(80,60)
      \thicklines
      \put(0,0){\circle*{4.5}}
      \put(40,0){\circle*{4.5}}
      \put(80,0){\circle*{4.5}}
      \put(40,60){\circle*{4.5}}
      \put(20,30){\circle*{4.5}}
      \put(60,30){\circle*{4.5}}
      \put(0,0){\line(1,0){80}}
      \put(0,0){\line(2,3){40}}
      \put(80,0){\line(-2,3){40}}
      \put(20,30){\line(1,0){40}}
      \put(20,30){\line(2,-3){20}}
      \put(60,30){\line(-2,-3){20}}
    \end{picture}
  \end{center}
  Clearly, $c(\Gamma)=3$, and since any two of 3-cliques intersect, $z_2(\Gamma)=5$. Since three side 3-cliques do not intersect, $z_3(\Gamma)=9$, implying $z_r(\Gamma)=3r$ for $r\ge 3$.
\end{example}

Let $A_\Gamma$ denote the right-angled Artin group over $\Gamma$. The following are proved by Gonzalez, Gutierrez and Yuzvinsky \cite{GGY} and Farber and Oprea \cite{FO}.

\begin{theorem}
  \label{RAAG}
  For a finite simple graph $\Gamma$ and $r\ge 2$,
  $$\cat(BA_\Gamma)=c(\Gamma)\quad\text{and}\quad\TC_r(BA_\Gamma)=z_r(\Gamma).$$
\end{theorem}

Now we are ready to prove the following, which shows us the reason why we need to consider the growth of $\TC_r(X)$ for all $r$ large enough, instead of all $r$.

\begin{proposition}
  \label{instability}
  Given an integer $n\ge 2$, there is a space $X$ such that $\TC_{n+1}(X)-\TC_n(X)>\cat(X)$ and $\TC_{r+1}(X)-\TC_r(X)=\cat(X)$ for all $r$ large enough.
\end{proposition}

\begin{proof}
  Let $\Gamma_n$ be the 1-skeleton of a simplicial complex obtained from an $n$-simplex by attaching an $n$-simplex on each $(n-1)$-face. Then $\Gamma_2$ is the graph in Example \ref{graph}. Let $v_0,\ldots,v_n$ be vertices of the base simplex of $\Gamma_n$, and let $w_0,\ldots,w_n$ be vertices of attached simplices, which are not in the base simplex. Then the vertex set of $\Gamma_n$ is $\{v_0,\ldots,v_n,w_0,\ldots,w_n\}$, and we may assume that maximal cliques of $\Gamma_n$ are $C_0=\{v_0,\ldots,v_n\}$ and $C_{i+1}=\{v_0,\ldots,\widehat{v_i},\ldots,v_n,w_i\}$ for $0\le i\le n$. Then, in particular, $c(\Gamma_n)=n+1$. For $r\le n$, any $r$ of $C_0,\ldots,C_{n+1}$ intersect, and $C_1\cap\cdots\cap C_{r-1}\cap(C_r-\{v_{r},\ldots,v_n\})=\emptyset$. Then $z_r(\Gamma_n)=(r-1)(n+1)+r$. For $r>n$, $C_1\cap\cdots\cap C_{n+1}\cap\underbrace{C_1\cap\cdots\cap C_1}_{r-n-1}=\emptyset$, implying $z_r(\Gamma_n)=r(n+1)$. Thus by Theorem \ref{RAAG}, the proof is complete.
\end{proof}

\begin{remark}
  We would like to thank John Oprea for letting us know the instability of the growth of $\TC_r(BA_{\Gamma_2})$, where Proposition \ref{instability} is its straightforward generalization. He also informed us, in a private communication, that Michael Farber and he also discovered the graph $\Gamma_n$ and showed further that the numerator polynomial $P(x)$ in Question \ref{question} is of degree $n+1$, implying that the degree of $P(x)$ can be arbitrary.
\end{remark}


\section{Module topological complexity}

In what follows, let $X$ be a simply-connected rational space of finite rational type. Let $(\Lambda V,d)$ be any Sullivan model of $X$, and let $K_r$ denote the kernel of the multiplication
$$\mu_r\colon(\Lambda V)^{\otimes r}\to\Lambda V,\quad v_1\otimes\cdots\otimes v_r\mapsto v_1\cdots v_r.$$
Then as in \cite{C}, $\TC_r(X)$ equals the least integer $n$ such that the projection
$$p_{r,n}\colon(\Lambda V)^{\otimes r}\to(\Lambda V)^{\otimes r}/K_r^{n+1}$$
has a differential graded algebra homotopy retraction. Here a homotopy retraction of $p_{r,n}$ means a retraction of a Sullivan model $(\Lambda V)^{\otimes r}\to(\Lambda V)^{\otimes r}\otimes\Lambda W$ for the projection $p_{r,n}$.

For $v\in V$ and $1\le i\le r$, let $v(i)=1\otimes\cdots\otimes 1\otimes v\otimes 1\otimes\cdots\otimes 1\in(\Lambda V)^{\otimes r}$ where $v$ is in the $i$-th position. Analogously to \cite[Lemma 2.3]{JMP}, one has:

\begin{lemma}
  \label{K}
  For $r\ge 2$,
  $$K_r=(v(i)-v(i+1)\mid v\in V,\,1\le i\le r-1).$$
\end{lemma}

In \cite{JMP}, the module topological complexity $\MTC(X)$ is defined, and we aim to consider its higher analog. The higher topological complexity $\TC_r(X)$ is, by definition, the sectional category of the diagonal map $\Delta_r\colon X\to X^r$. In \cite{C,C1}, the sectional category for rational spaces are studied, and the module sectional category, which generalizes the module topological complexity, is introduced. Now we specialize the module sectional category to define the higher module sectional category.

\begin{definition}
  The $r$-th module topological complexity $\MTC_r(X)$ is defined to be the least integer $n$ such that the projection $p_{r,n}$ has a differential $(\Lambda V)^{\otimes r}$-module homotopy retraction.
\end{definition}

For a space $Y$ and a field $\Bbbk$, let $\zcl_r(Y;\Bbbk)$ be the greatest $n$ such that there are elements $x_1,\ldots,x_n$ in the kernel of the product $H^*(Y;\Bbbk)^{\otimes r}\to H^*(Y;\Bbbk)$ such that $x_1\cdots x_n\ne 0$. Then one has:

\begin{proposition}
  \label{zcl<MTC<TC}
  For $r\ge 2$,
  $$\zcl_r(X;\Q)\le\MTC_r(X)\le\TC_r(X).$$
\end{proposition}

\begin{proof}
  Suppose $\MTC_r(X)=n$. Then the projection $p_{r,n}$ is injective in cohomology, and so there are no $x_0,\ldots,x_n\in K_r$ such that the product $x_0\cdots x_n$ is a cocycle representing a non-trivial cohomology class of $X$. Consider the long exact sequence associated with a short exact sequence of cochain complexes
  $$0\to K_r\to(\Lambda V)^{\otimes r}\xrightarrow{\mu_r}\Lambda V\to 0.$$
  Then since $(\mu_r)_*\colon H^*((\Lambda V)^{\otimes r})\to H^*(\Lambda V)$ is surjective, its kernel is $H^*(K_r)$, implying $\zcl_r(X;\Q)\le n=\MTC_r(X)$. The inequality $\MTC_r(X)\le\TC_r(X)$ follows at once from the definition.
\end{proof}


\section{Proofs of main results}

\begin{lemma}
  \label{lower bound}
  For $r\ge 2$,
  $$\cl(X;\Q)\le\MTC_{r+1}(X)-\MTC_r(X).$$
\end{lemma}

\begin{proof}
  Let $(\Lambda V,d)$ be a Sullivan model for $X$, and suppose $\cl(X;\Q)=n$. Then there are cocycles $a_1,\ldots,a_n\in\Lambda V$ such that the product $a=a_1\cdots a_n$ represents a non-trivial cohomology class. For each $i$, let
  $$\bar{a}_i=\underbrace{1\otimes\cdots\otimes 1}_r\otimes a_i-\underbrace{1\otimes\cdots\otimes 1}_{r-1}\otimes a_i\otimes 1\in(\Lambda V)^{\otimes(r+1)}$$
  and let $\bar{a}=\bar{a}_1\cdots\bar{a}_n$. Define a map
  $$f\colon(\Lambda V)^{\otimes r}\to(\Lambda V)^{\otimes(r+1)},\quad x\mapsto (x\otimes 1)\bar{a}.$$
  Since $\bar{a}$ is a cocycle, $f$ is a differential graded $(\Lambda V)^{\otimes r}$-module map.

  Let $C$ be a cochain complex of finite type over a field. If $c_1,\ldots,c_k\in C$ are cocycles which are not coboundaries, then there is a decomposition of a cochain complex
  $$C=\langle c_1,\ldots,c_k\rangle\oplus D$$
  for some subcomplex $D$ of $C$. Hence since $a_{i_1}\cdots a_{i_k}$ for $1\le i_1<\cdots<i_k\le n$ are cocycles which are not coboundaries, there is a decomposition of a cochain complex
  $$\Lambda V=\langle a\rangle\oplus W$$
  for some subcomplex $W$ of $\Lambda V$ such that $1\in W$ and $a_{i_1}\cdots a_{i_k}\in W$ for $1\le i_1<\cdots<i_k\le n$ with $k<n$. Thus one gets a decomposition of a differential graded $((\Lambda V)^{\otimes r}\otimes 1)$-module
  $$(\Lambda V)^{\otimes(r+1)}=((\Lambda V)^{\otimes r}\otimes 1)\cdot\langle\bar{a}\rangle\oplus((\Lambda V)^{\otimes r}\otimes W).$$
  Therefore there is a differential graded $(\Lambda V)^{\otimes r}$-module retraction of $f$
  $$q\colon(\Lambda V)^{\otimes(r+1)}\to(\Lambda V)^{\otimes r},\quad q(\bar{a})=1,\quad q((\Lambda V)^{\otimes r}\otimes W)=0.$$

  It is proved in \cite{FKS} that $\zcl_{r+1}(X;\Bbbk)\ge\zcl_r(X;\Bbbk)+\cl(X;\Bbbk)$ for any field $\Bbbk$. Then by Proposition \ref{zcl<MTC<TC}, $\MTC_{r+1}(X)=m+n$ for some non-negative integer $m$, and so there is a differential graded $(\Lambda V)^{\otimes(r+1)}$-module homotopy retraction $q'$ of the projection $p_{r+1,m+n}\colon(\Lambda V)^{\otimes(r+1)}\to(\Lambda V)^{\otimes(r+1)}/K_{r+1}^{m+n+1}$. On the other hand, since $\bar{a}$ belongs to $K_{r+1}^n$, there is a commutative diagram of differential graded $(\Lambda V)^{\otimes r}$-modules
  $$\xymatrix{(\Lambda V)^{\otimes r}\ar[r]^f\ar[d]_{p_{r,m}}&(\Lambda V)^{\otimes(r+1)}\ar[d]^{p_{r+1,m+n}}\\
  (\Lambda V)^{\otimes r}/K_r^{m+1}\ar[r]^{\bar{f}}&(\Lambda V)^{\otimes(r+1)}/K_{r+1}^{m+n+1}.}$$
  Then $q\circ q'\circ\bar{f}$ is a differential graded $(\Lambda V)^{\otimes r}$-module homotopy retraction of the projection $p_{r,m}$. Thus $\MTC_r(X)+\cl(X;\Q)\le m+n=\MTC_{r+1}(X)$.
\end{proof}

\begin{lemma}
  \label{upper bound}
  For $r\ge 2$,
  $$\MTC_{r+1}(X)-\MTC_r(X)\le2\cat(X).$$
\end{lemma}

\begin{proof}
  Since there are obvious inequalities $\TC_r(X)\le \cat(X^r)\le r\cat(X)$, it follows from Proposition \ref{zcl<MTC<TC} that
  $$\MTC_r(X)\le r\cat(X).$$
  There is a homotopy pullback
  $$\xymatrix{\ast\ar[r]\ar[d]&X\ar[d]^{\Delta_r}\\
  X^{r-1}\ar[r]^j&X^r}$$
  where $j(x_1,\ldots,x_{r-1})=(x_1,\ldots,x_{r-1},*)$ for $(x_1,\ldots,x_{r-1})\in X^{r-1}$. Then since $X$ is a rational space, it follows from \cite[Proposition 5.5]{FGKV} and \cite[Theorem 30.2]{FHT} that
  \begin{equation}
    \label{MTC lower bound}
    \MTC_r(X)\ge\cat(X^{r-1})=(r-1)\cat(X).
  \end{equation}
  Thus $\MTC_{r+1}(X)-\MTC_r(X)\le(r+1)\cat(X)-(r-1)\cat(X)=2\cat(X)$, as claimed.
\end{proof}

Now we are ready to prove Theorems \ref{MTC cup} and \ref{odd}.

\begin{proof}
  [Proof of Theorem \ref{MTC cup}]
  Combine Lemmas \ref{lower bound} and \ref{upper bound}.
\end{proof}

\begin{proof}
  [Proof of Corollary \ref{generating MTC}]
  One has $\MTC_r(X)\le\TC_r(X)\le r\cat(X)$. Also $\MTC_r(X)+\cl(X;\Q)\le\MTC_{r+1}(X)$ by Theorem \ref{MTC cup}. Then the statement follows from the proof of \cite[Lemma 1]{FKS} where it is shown that $P(1)=\cat(X)$.
\end{proof}

\begin{proof}
  [Proof of Theorem \ref{odd}]
  Let $(\Lambda V,d)$ be the minimal model of $X$. By Lemma \ref{K},
  $$K_r^{n+1}=(v(i)-v(i+1)\mid v\in V,\,1\le i\le r-1)^{n+1}$$
  where $v(i)$ is as in Lemma \ref{K}. By assumption, $V$ has no even degree element, and so $K_r^{n+1}=0$ for $n=(r-1)\dim V$. On the other hand, it follows from \cite[Example 6, p. 389]{FHT} that $\cat(X)=\dim V$. Then one gets
  $$\TC_r(X)\le(r-1)\cat(X).$$
  On the other hand, by \eqref{MTC lower bound}, $\MTC_r(X)\ge(r-1)\cat(X)$. Thus the proof is complete by Proposition \ref{zcl<MTC<TC}.
\end{proof}


\end{document}